\documentclass[12pt]{article}
\usepackage{amsmath,amssymb,amscd}
\def\complex{\mathbb{C}}

\def\bilin#1#2{\left\langle#1,\,#2\right\rangle}

 \newtheorem{theorem}{Theorem}[section]
 \newtheorem{proposition}[theorem]{Proposition}
 \newtheorem{corollary}[theorem]{Corollary}
 \newtheorem{lemma}[theorem]{Lemma}
 
 \newtheorem{example}[theorem]{Example}
 \newtheorem{remark}[theorem]{Remark}

\numberwithin{equation}{section}

\newenvironment{proof}{\smallskip\par{\sc Proof.}\enspace}%
 {{\unskip\nobreak\hfil\penalty50\hskip2em
          \hbox{}\nobreak\hfil{\rule[-1pt]{5pt}{10pt}}
          \parfillskip=0pt\finalhyphendemerits=0
          \par\medskip}} 

\begin{document}

\vspace*{.3in}

\begin{center}
\LARGE
{\sf Convergences of Alternating Projections \\ \medskip
in ${\rm CAT}(\kappa)$ Spaces}
\end{center}

\bigskip
\begin{center}
Byoung Jin Choi\\
Department of Mathematics\\
Sungkyunkwan University\\
Suwon 16419, Korea\\
\texttt{E-Mail: choibj@skku.edu}\\ \medskip

Un Cig Ji\\
Department of Mathematics \\
Research Institute of Mathematical Finance \\
Chungbuk National University \\
Cheongju 28644, Korea \\
\texttt{E-Mail: uncigji@chungbuk.ac.kr}\\ \medskip
and \\[5pt]
Yongdo Lim\\
Department of Mathematics\\
Sungkyunkwan University\\
Suwon 16419, Korea\\
\texttt{E-Mail:ylim@skku.edu}\\
\end{center}

\begin{abstract}
We establish the asymptotic regularity and the $\Delta$-convergence
of the sequence constructed by the alternating projections to closed
convex sets in a ${\rm CAT}(\kappa)$ space with $\kappa > 0$.
Furthermore, the strong convergence of the alternating von Neumann
sequence is presented under  certain regularity or compactness.
\end{abstract}


\section{Introduction}

Alternating projection algorithm is one of the most simple and important algorithm
for computing a point in the intersection of some convex sets, which is called the convex feasibility problem.
More precisely, for given two closed convex subsets $A$ and $B$ with the corresponding projections $P_A$ and $P_B$,
the \textit{alternating projection method} produces the sequence:
\begin{equation*}
b_0=x_0,
\quad
a_n=P_{A}(b_{n-1}),\quad b_{n}=P_{B}(a_{n}),\qquad n \in \mathbb{N},
\end{equation*}
where $x_0$ is a given starting point.
The alternating projection algorithm for the case of
closed subspaces $A$ and $B$ of a Hilbert space
was introduced by von Neumann \cite{Neumann50},
and so it is also called the von Neumann's alternating projection algorithm.
In this case, the sequences $\{a_n\}$ and $\{b_n\}$
are referred to as von Neumann sequences,
and $\{b_0,a_1,b_1,a_2,b_2,\cdots\}$ is referred to as an alternating von Neumann sequence,
which is denoted by the sequence $\{x_n\}_{n=0}^\infty$, i.e.,
\begin{equation}\label{eqn:alternating projections x-n}
x_0=b_0,
\quad x_{2m}=b_m=P_{B}(a_m),\quad x_{2m-1}=a_m=P_{A}(b_{m-1}),\qquad m \in \mathbb{N}.
\end{equation}
In 1933, von Neumann proved that the alternating projections
defined as in \eqref{eqn:alternating projections x-n}
converges in norm to $P_{A\cap B} (x_0)$,
when $A$ and $B$ are closed subspaces of a Hilbert space
(see \cite{Neumann50,BMR04}).
In \cite{Bregman65}, Br\`{e}gman proved that the alternating projections
defined as in \eqref{eqn:alternating projections x-n}
converges weakly to a point in the intersection of closed convex sets of a Hilbert space,
if the intersection is non-empty.
In \cite{Hundal04}, Hundal proved that the alternating projections
between two closed convex intersecting sets does not always converge in norm,
by providing a sequence of alternating projections which converges weakly,
but does not converges in norm.

The convex feasibility problem has been studied by many authors,
e.g.,
\cite{Neumann50,Bregman65,BB96,Combettes97,BMR04,Hundal04,BSS12,ARLAN15}
and references cited therein. In particular, Bauschke and Borwein
\cite{BB96} studied the alternating projection algorithm to solve
the convex feasibility problem. In \cite{BB96}, we also can find
several applications of the alternating projection algorithm to the
best approximation theory,  discrete and continuous image
reconstruction, and subgradient algorithms.
In \cite{Zarikian2006}, Zarikian used the alternating projection algorithm to
solve a variety of operator-theoretic
problems, e.g., deciding complete positivity, computing completely bounded norms, etc.
By many authors, the
alternating projection method has been extended to general metric
spaces for solving the convex feasibility problem. In \cite{BSS12},
Ba\v{c}\'{a}k et al. studied the sequence of alternating projections
in a ${\rm CAT}(0)$ space or alternatively Hadamard space and then
proved that the sequence converges weakly (equivalently,
$\Delta$-converges) to a point in the intersection of closed convex
subsets of the CAT$(0)$ space.

Recently, in \cite{ARLAN15},
by developing unified treatment of convex minimization problems,
Ariza-Ruiz et al. studied the asymptotic behavior of the sequence
constructed by the iterative method for a finite family of firmly nonexpansive maps
in the setting of $p$-uniformly convex spaces,
and then the asymptotic regularity has been applied to study the common fixed points
of the finite family of firmly nonexpansive maps and also to study the  convex feasibility problem.
In \cite{ARLAN15}, the ${\rm CAT}(0)$ spaces and ${\rm CAT}(\kappa)$ spaces with $\kappa>0$
were considered as examples of $p$-uniformly convex spaces.

The notion of the ${\rm CAT}(\kappa)$ spaces \cite{Bhatia07,BH99},
a generalization of Riemannian manifolds of sectional curvature bounded above,
has been introduced by Gromov \cite{Gromov87},
where ${\rm CAT}$ stands for Cartan-Alexandrov-Toponogov.
The group of unitary matrices and the $n$-dimensional sphere have nonnegative sectional curvature (see \cite{Bhatia07}),
indeed, the group $SU(2)$ is a ${\rm CAT}(1)$ space (see Example \ref{example:n-dimensional sphere and SU2 is cat 1 space} (ii)).
Also, the vector (pure) state space of the space of $2 \times 2$ matrices $M_{2}(\mathbb{C})$
is a ${\rm CAT}(1)$ space (see Example \ref{example:vector state space is cat 1 space}).
Every ${\rm CAT}(0)$ space is ${\rm CAT}(\kappa)$ space for all $\kappa > 0$,
indeed, a ${\rm CAT}(\kappa)$ space is a ${\rm CAT}(\kappa')$ space for all $\kappa, \kappa' \in \mathbb{R}$ with $ \kappa \le \kappa'$,
and Hilbert spaces and classical hyperbolic spaces are
typical examples of ${\rm CAT}(0)$ spaces.
Therefore, the study of ${\rm CAT}(\kappa)$ spaces is getting more interesting
to overcome several difficulties appearing in the estimation theory
by using geometric structures.

Main purpose of this paper is to study the sequence constructed by
the alternating projection method in ${\rm CAT}(\kappa)$ spaces with $\kappa > 0$.
Then we prove the asymptotic regularity and the $\Delta$-convergence of the alternating von Neumann sequence.
Also, we prove the strong convergence of the sequence
under certain regularity or compactness condition on closed convex sets.
By comparing the results in \cite{ARLAN15},
it is emphasized that the main results in this paper
are proved without any non-expensiveness condition.

This paper is organized as follows.
In Section 2, we review briefly the basic notions in the setting of ${\rm CAT}(\kappa)$ spaces
and the $\Delta$-convergence in ${\rm CAT}(\kappa)$ spaces.
In Section 3, we first prove the asymptotic regularity of the sequence
constructed by the alternating projection method (Theorem \ref{thm:rate of asymptotic regularity of the alternating projection method}),
and secondly we prove the $\Delta$-convergence of the alternating von Neumann sequence
in ${\rm CAT}(\kappa)$ spaces (Theorem \ref{thm:convergence of Al.pro}).
Furthermore, we prove the strong convergence of the alternating von Neumann sequence
by assuming certain regularity or compactness condition on convex sets
(Corollary \ref{coro:strong convergence of Al.pro}).

\section{Preliminary}

\subsection{CAT$(\kappa)$ spaces}

Let $(M, d)$ be a metric space.
A \textit{geodesic (path)} joining $x\in M$ and $y \in M$ is a map $\gamma:[0,1]\rightarrow M$
satisfying that $\gamma(0)=x$, $\gamma(1)=y$ and
$d(\gamma(t_1),\gamma(t_2))=d(x,y)|t_1-t_2|$ for all $t_1, t_2 \in [0,1]$.
The image of the geodesic $\gamma$ joining $x$ and $y$ is called a \textit{geodesic segment} joining $x$ and $y$.
If for any $x,y\in M$, there exists a unique geodesic joining $x$ and $y$,
then the unique geodesic is denoted by $[x,y]$.

A metric space $(M,d)$ is called a \textit{geodesic space} (\textit{$D$-geodesic space})
if for any two points $x,y \in M$ (for any two points $x,y\in M$ with $d(x,y)<D$),
there exists a geodesic $\gamma$ joining $x$ and $y$.
A subset $C$ of $M$ is said to be \textit{convex} if
\begin{itemize}
  \item [\textbf{(C)}] any two points $x,y \in C$ can be joined by a geodesic in $M$
and the geodesic segment of every such geodesic is contained in $C$.
\end{itemize}

\begin{remark}
\upshape
In some literature,
a subset $C$ of $M$ satisfying the condition \textbf{(C)}
is said to be totally convex.
In this case, a subset $C$ of $M$ is said to be convex
if any two points of $C$ are joined by a geodesic belonging to $M$.
\end{remark}

The $n$-dimensional sphere $\mathbb{S}^n$ is the set
\[
\{x=(x_1,x_2,\cdots,x_{n+1})\in \mathbb{R}^{n+1}\,|\,\bilin{x}{x}=1\},
\]
where $\bilin{\cdot}{\cdot}$ is the Euclidean scalar product.
Let $\rho: \mathbb{S}^n \times \mathbb{S}^n \rightarrow \mathbb{R}$ be the function that
assigns to each pair $(x,y) \in \mathbb{S}^n \times \mathbb{S}^n$ the unique real number $\rho(x,y) \in [0,\pi]$ such that
\[
\cos \rho(x,y)=\bilin{x}{y}.
\]
Then it is well known fact that $(\mathbb{S}^n,\rho)$ is a geodesic metric space, and
if $\rho(x,y) <\pi$ for $x,y \in \mathbb{S}^n$,
then there is only one geodesic segment joining $x$ and $y$.

From now on, we always assume that $\kappa > 0$, and put $D_{\kappa}:=\pi/\sqrt{\kappa}$.
Then the \textit{model space} $M_{\kappa}^{n}$ is the metric space obtained from $(\mathbb{S}^{n},\rho)$
by multiplying the distance function by the constant $1/\sqrt{\kappa}$.
We use the same symbol $\rho$ for the distance function of $M_{\kappa}^{n}$.
Then it is clear that $M_{\kappa}^{n}$  is a geodesic metric space.
Note that there is a unique geodesic segment joining $x,y \in M_{\kappa}^{n}$ if and only if $\rho(x,y)<D_{\kappa}$.

Let $(M,d)$ be a geodesic metric space.
A \textit{geodesic triangle} $\Delta:=\Delta(x,y,z)$ in the metric space $M$
consists of three points in $M$ and three geodesic segments joining each pair of points.
For a geodesic triangle $\Delta=\Delta(x,y,z) \subseteq M$,
a geodesic triangle $\overline{\Delta}=\Delta(\overline{x},\overline{y},\overline{z}) \subseteq M_{\kappa}^{2}$
is called a \textit{comparison triangle} for $\Delta$ if
\[
d(x,y)=\rho(\overline{x},\overline{y}),\quad d(x,z)=\rho(\overline{x},\overline{z})\quad \text{and}
\quad d(y,z)=\rho(\overline{y},\overline{z}).
\]
A point $\overline{p} \in [\overline{x},\overline{y}] \subseteq \overline{\Delta}$
is called a \textit{comparison point} for $p \in [x,y] \subseteq \Delta$ if
$d(p,x)=\rho(\overline{p},\overline{x})$.
Note that for a geodesic triangle $\Delta=\Delta(x,y,z) \subseteq M$,
if the perimeter of $\Delta$ is (strictly) less than $2D_k$, i.e., $d(x,y)+d(y,z)+d(z,x) < 2 D_{\kappa}$,
then a comparison triangle for $\Delta$ always exists (see \cite{BH99}).
For a geodesic triangle $\Delta=\Delta(x,y,z) \subseteq M$ of perimeter less than $2D_k$,
we say that $\triangle$ satisfies \textit{${\rm CAT}(\kappa)$ inequality}
if for any $p,q\in\triangle$ and their comparison points $\overline{p},\overline{q}\in\overline{\triangle}$,
it holds that
\begin{equation}\label{eqn:CAT(kappa)--inequality}
d(p, q)\leq \rho(\overline{p},\overline{q}).
\end{equation}

A metric space $(M,d)$ is called a \textit{CAT$(\kappa)$ space}
if $(M,d)$ is a $D_k$-geodesic space and all geodesic triangles in $M$ of perimeter less than $2D_{\kappa}$
satisfy the ${\rm CAT}(\kappa)$ inequality.

For a non-empty subset $C$ of a metric space $(M, d)$,
the \textit{diameter} of $C$ is defined by
\begin{align*}
{\rm diam}(C)&=\sup\{d(x,y)\,;\,x,y\in C\}.
\end{align*}
If $(M,d)$ is a ${\rm CAT}(\kappa)$ space with ${\rm diam}(M)<D_{\kappa}/2$,
then for any geodesic $\gamma:[0,1]\rightarrow M$
with $\gamma(0)=x$ and $\gamma(1)=y$, any $z\in M$ and $t \in [0,1]$,
there exists a constant $c_{M} \in (0,1)$ such that the following inequality holds:
\begin{equation}\label{eqn:convexity for CAt(k) space}
d(z,\gamma(t))^2 \le (1-t)d(z,x)^2+t d(z,y)^2-c_{M} t(1-t)d(x,y)^2,
\end{equation}
(see \cite{Kuwae14,Ohta07}).
\begin{example}\label{example:n-dimensional sphere and SU2 is cat 1 space}
\rm
\begin{itemize}
  \item [{\rm(i)}] A complete simply connected Riemannian manifold with constant sectinal curvature $\le \kappa$
  is a ${\rm CAT}(\kappa)$ space: the $n$-dimensional sphere is a ${\rm CAT}(1)$ space (see \cite{Gromov2001}).
  \item [{\rm(ii)}] The group $SU(2)$, consisting of $2 \times 2$ unitary matrices with determinant $1$ is a ${\rm CAT}(1)$ space.
Indeed, there are bijective functions from $SU(2)$ to the $3$-dimensional sphere $\mathbb{S}^3$.
Furthermore, for a compact, simply connected Lie group $G$,
if $G$ admits a left invariant Riemannian structure with strictly positive sectional curvature,
then $G$ is Lie isomorphic with $SU(2)$ (see \cite[Theorem 2.1]{Wallach72}).
\end{itemize}

\end{example}

\begin{example}\label{example:vector state space is cat 1 space}
\rm
Let $H$ be a Hilbert space and $\mathcal{B}(H)$ the space of all bounded linear operators from $H$ into itself.
A linear functional $\tau$ from $\mathcal{B}(H)$ into $\complex$ is said to be \textit{positive}
if $\tau(X^{*}X)\geq0$ for any $X \in \mathcal{B}(H)$.
The linear functioanl $\tau$ is \textit{normalized} if
$\tau(I)=1$, where $I$ is the identity map.
A normalized positive linear functional $\tau$ from $\mathcal{B}(H)$ into
$\complex$ is called a \textit{state}.
A state $\tau$ is said to be \textit{normal} if
\[
\sup_{\alpha}\tau(X_\alpha)=\tau(\sup_{\alpha}X_\alpha)
\]
for any positive bounded net $\{X_\alpha\}$.
In particular, for each unit vector $\xi$, the state $\tau_{\xi}: \mathcal{B}(H)\ni X\mapsto \bilin{X\xi}{\xi}\in \mathbb{C}$ is said to be the
\textit{vector state} on $\mathcal{B}(H)$ determined by $\xi$.
Let $\mathcal{S}$ be the space of all normal states on $\mathcal{B}(H)$.
Note that the extreme points $\partial\mathcal{S}$ of $\mathcal{S}$ consists of all
vector states $\tau_{\xi}$ with unit vector $\xi$ in $H$.
If ${\mathrm{dim}}H=2$, then $\mathcal{B}(H)$ becomes the vector space of all $2 \times 2$ matrixes $M_{2}(\mathbb{C})$.
Also we can identify the normal state space $\mathcal{S}$ of $M_{2}(\mathbb{C})$
with the convex set of all positive trace one matrix in $M_{2}(\mathbb{C})$,
and we can identify $\partial\mathcal{S}$ with $\mathbb{S}^2$ (see \cite{AHS80,AS2001}).
Therefore, the vector (pure) state space of $M_{2}(\mathbb{C})$ is a ${\rm CAT}(1)$ space.
\end{example}

For a non-empty subset $C$ of a metric space $(M, d)$,
the \textit{distance function} of $C$ is defined by
\begin{align*}
d(x,C)&=\inf\{d(x,c)\,;\, c \in C\},\qquad \text{for}\,\, x \in M.
\end{align*}

\begin{proposition}{\rm \cite{BH99,EFL09,ARLAN15}}\label{pro:properties of metric projection}
Let $(M,d)$ be a complete ${\rm CAT}(\kappa)$ space and $x \in M$ be given. Let $C \subseteq M$ be a non-empty closed convex set
such that $d(x,C)<D_{\kappa}/2$.
Then for given $x \in M$, there exists a unique point $P_C(x) \in C$ such that
  \begin{equation}\label{eqn:for def of P-C}
  d(x,P_C(x))=d(x,C).
  \end{equation}
\end{proposition}

Let $(M,d)$ be a complete ${\rm CAT}(\kappa)$ space with ${\rm diam}(M)<D_{\kappa}/2$
and $C$ be a non-empty closed convex subset of $M$.
Then from Proposition \ref{pro:properties of metric projection}
we define the map
\[
P_{C}:M\ni x\longmapsto P_C(x)\in C,
\]
where for each $x\in M$, $P_C(x)$ is the unique element in $C$ satisfying \eqref{eqn:for def of P-C}.
The map $P_C$ is called the \textit{(metric) projection} onto $C$.
For more detailed study of ${\rm CAT}(\kappa)$ spaces, we refer to \cite{BH99,EFL09}.

Now, we review some notions for $p$-uniformly convex spaces.
Let $1<p<\infty$. Then a metric space $(M,d)$ is called \textit{$p$-uniformly convex} with parameter $c>0$
if $(M,d)$ is a geodesic space and for any $x,y,z\in M$ and $t\in[0,1]$,
\begin{equation}\label{eqn:inequality for p-uniformly convex}
d(z,\gamma_{x,y}(t))^p
\le (1-t)d(z,x)^p+td(z,y)^p-\frac{c}{2}t(1-t)d(x,y)^p,
\end{equation}
where $\gamma_{x,y}$ is a geodesic joining $x$ and $y$ such that $\gamma_{x,y}(0)=x$ and $\gamma_{x,y}(1)=y$.

It is well-known that every $L^p(\Omega,\mu)$ ($1<p<\infty$) over a measure space $(\Omega,\mu)$ is $p$-uniformly convex.
Also, for $\kappa>0$, every ${\rm CAT}(\kappa)$ space is $p$-uniformly convex, see \cite{Ohta07}.

Let $C$ be a non-empty subset of a $p$-uniformly convex space $(M,d)$.
A map $T:C\to X$ is said to be \textit{firmly nonexpansive} if
\begin{equation}\label{eqn:inequality for firmly nonexpansive map}
d(Tx,Ty)
\le d(\gamma_{x,Tx}(t),\gamma_{y,Ty}(t))
\end{equation}
for all $x,y\in C$ and $t\in[0,1)$.
The notion of firmly nonexpansive map has been introduced by Browder \cite{Browder67} (see, also \cite{Bruck73}).
For more study of firmly nonexpansive maps, we refer to \cite{ARLAN15}.

For any non-empty closed convex subset $C$ of a complete ${\rm CAT}(0)$ space $(M,d)$,
the metric projection map $P_C:M\to C\subset M$ is firmly nonexpansive.
But, in general, if $\kappa>0$,
then the metric projection map $P_C$ for a non-empty closed convex subset $C$
of a complete ${\rm CAT}(\kappa)$ space need not be nonexpansive.
However, more interesting examples of firmly nonexpansive map can be found in \cite{ARLAN15}.

Let $C$ be a non-empty subset of a metric space $(M,d)$.
A mapping $T:C\to M$ is said to satisfy property (P1)
if ${\rm Fix}(T)\ne \emptyset$ and there exist $\ell,\beta>0$ such that
\begin{equation}\label{eqn:inequality for property (P1)}
d(Tx,u)^\ell\le d(x,u)^\ell-\beta d(Tx,x)^\ell
\end{equation}
for all $x\in C$ and $u\in{\rm Fix}(T)$ (see Definition 2.5 in \cite{ARLAN15}).

\begin{proposition}{\rm \cite{ARLAN15}}\label{propo:inequality of projection}
Let $(M,d)$ be a complete ${\rm CAT}(\kappa)$ space with ${\rm diam}(M)<D_{\kappa}/2$
and $C \subseteq M$ be a non-empty closed convex set.
Then for all $x \in M$ and $z\in C$, it holds that
\begin{equation}\label{eqn:metric inequality of projection}
d(z,P_{C}(x))^{2}+c_{M}d(x,P_{C}(x))^{2} \le d(x,z)^{2},
\end{equation}
where $c_{M}$ is given as in \eqref{eqn:convexity for CAt(k) space}.
\end{proposition}

Let $(M,d)$ be a complete ${\rm CAT}(\kappa)$ space with ${\rm diam}(M)<D_{\kappa}/2$
and $C$ be a non-empty closed convex subset of $M$.
Then from \eqref{eqn:metric inequality of projection},
it is obvious that the projection map $P_C:M\to C\subset M$ satisfies the property (P1).

\subsection{$\Delta$-convergence in ${\rm CAT}(\kappa)$ spaces}

For our purpose to study the convergence of the alternating projections
in a metric space without any linear structure,
motivated of the studies by Bregman \cite{Bregman65} and Hundal \cite{Hundal04},
we need a modification of the notion of the weak convergence in a normed space.
One of such modification is called the $\Delta$-converge which was first introduced by Lim \cite{Lim76}.
In \cite{KP08}, the authors studied the $\Delta$-convergence in CAT$(0)$ spaces
and the convergence was applied to study the fixed point theory.
In \cite{EFL09}, the authors studied the $\Delta$-convergence in ${\rm CAT}(\kappa)$ spaces
and the authors proved that, in ${\rm CAT}(0)$ spaces,
the $\Delta$-convergence coincides with the modified $\phi$-convergence introduced
by Sosov \cite{Sosov04} (see \cite[Proposition 5.2]{EFL09}).
Those convergences are generalizations to ${\rm CAT}(0)$ spaces
of the notion of the weak convergence in Hilbert spaces.
In fact, the two different notions of convergence introduced by Sosov \cite{Sosov04}
coincide with the $\Delta$-convergence and the weak convergence in Hilbert spaces.

Let $(M, d)$ be a complete ${\rm CAT}(\kappa)$ space and $\{x_n\}\subseteq M$ be a (bounded) sequence.
For a given point $x \in M$, set
\[
r(x,\{x_n\})=\limsup_{n \rightarrow \infty }d(x, x_n),
\]
and then, the \textit{asymptotic radius} $r(\{x_n\})$ of $\{x_n\}$ is defined by
\[
r(\{x_n\})=\inf_{x \in M}r(x,\{x_n\}).
\]
and  the \textit{asymptotic center} $A(\{x_n\})$ of $\{x_n\}$ is defined as
\[
A(\{x_n\}):=\left\{x \in M \,|\, r(x,\{x_n\})=r(\{x_n\})\right\}.
\]
Note that $z \in A(\{x_n\})$ if and only if
$\displaystyle\limsup_{n \rightarrow \infty}d(z,x_{n})\le \limsup_{n \rightarrow \infty}d(x,x_{n})$ for any $x\in M$.

Now, we recall the notion of $\Delta$-convergence in ${\rm CAT}(\kappa)$ spaces.
Let $x \in M$. A sequence $\{x_n\}$ is said to \textit{$\Delta$-converge}
to $x$ if for any subsequence $\{x_{n_k}\}$ of $\{x_n\}$, the point $x$ is the unique asymptotic center of $\{x_{n_k}\}$,
and then $x$ is called the \textit{$\Delta$-limit} of $\{x_n\}$.
A point $x$ in $M$ is called a \textit{$\Delta$-cluster point} of a sequence $\{x_n\}$
if there exists a subsequence $\{x_{n_k}\}$ of $\{x_n\}$ such that $\{x_{n_k}\}$ $\Delta$-converges to $x$.

\begin{remark}
\upshape
In some literatures, the authors used the notion of weak convergence
instead of the $\Delta$-convergence, see \cite{BSS12}.
\end{remark}

\begin{proposition}{\rm \cite{EFL09,HFLL12}}\label{properties of bounded sequence}
Let $M$ be a complete ${\rm CAT}(\kappa)$ space
and $\{x_n\}\subseteq M$ be a sequence with $r(\{x_n\})<D_{\kappa}/2$.
Then the following facts hold.
\begin{itemize}
  \item [\rm{(i)}] $A(\{x_n\})$ has only one point.
  \item [\rm{(ii)}] $\{x_n\}$ has a $\Delta$-convergent subsequence,
  i.e., $\{x_n\}$ has a $\Delta$-cluster point $x \in M$.
\end{itemize}
\end{proposition}

\begin{proposition}{\rm \cite{EFL09,HFLL12}}\label{propo:delta limit of xn is contained conv}
Let $M$ be a complete ${\rm CAT}(\kappa)$ space and let $z \in M$. If a sequence $\{x_n\}\subseteq M$ satisfies that $r(z,\{x_n\})<D_k /2$
and that $\{x_n\}$ $\Delta$-converges to $x\in M$, then
\begin{equation*}
x \in \bigcap_{k=1}^{\infty}\overline{\rm{conv}}\left( \{x_k,x_{k+1},\cdots\} \right),
\end{equation*}
where $\overline{\rm{conv}}(A)=\bigcap \{B \subseteq M \,|\, A \subseteq B \,\,\text{and}\,\,B\,\text{is closed and convex}\}$,
and
\begin{equation*}
d(x,z) \le \liminf_{n \rightarrow \infty}d(x_n, z).
\end{equation*}

\end{proposition}

\begin{remark}\label{rem:weakly lsc at x for dC}
\rm
Under the same assumptions as in Proposition \ref{propo:delta limit of xn is contained conv},
if $r(z,\{x_n\})<D_k/2$ for any $z$ in a subset $C$ of $M$,
then
\begin{equation*}
d(x,z) \le \liminf_{n \rightarrow \infty}d(x_n, z)\quad \text{for all} ~~z \in C.
\end{equation*}
\end{remark}

\section{Alternating Projections in CAT($\kappa$) Spaces}

For our study, we review the notion of Fej\'{e}r monotone sequences.
Let $C$ be a non-empty subset of a metric space $(M,d)$ and let $\{x_n\}$ be a sequence in $M$.
Then $\{x_n\}$ is said to be \textit{Fej\'{e}r monotone} with respect to $C$
if for any $z \in C$ and $n \in \mathbb{N}$, it holds that
\[
d(x_{n+1},z)\le d(x_n,z).
\]

A sequence $\{x_n\}$ converges \textit{linearly} to a point $x \in M$
if there exists $K \ge 0$ and $\alpha \in [0,1)$ such that
\[
d(x_n, x) \le K \alpha^n, \quad n \in \mathbb{N}.
\]
In this case, $\alpha$ is called a \textit{rate} of the linear convergence.

The following proposition is from Proposition 3.3 in \cite{BSS12}.

\begin{proposition}\label{propo:properties of Fejer monotone sequence}
Let $C$ be a non-empty closed convex subset of a complete metric space $(M,d)$
and let $\{x_n\}$ be a sequence in $M$.
Suppose that $\{x_n\}$ is Fej\'{e}r monotone with respect to $C$.
Then the following properties hold.
\begin{itemize}
  \item [\rm{(i)}] $\{x_n\}$ is a bounded sequence.
  \item [\rm{(ii)}] $d(x_{n+1},C) \le d(x_n,C)$ for all $n \in \mathbb{N}$.
  \item [\rm{(iii)}] $\{x_n\}$ converges to some $x \in C$ if and only if $d(x_n,C)\rightarrow 0$ as $n\rightarrow \infty$.
  \item [\rm{(iv)}] If there exists $\beta \in [0,1)$ such that $d(x_{n+1},C) \le \beta d(x_n,C)$ for each $n \in \mathbb{N}$, then
  $\{x_n\}$ converges linearly to some point $x \in C$.
\end{itemize}
\end{proposition}
\begin{proof}
The proofs of (i) and (ii) are obvious
and the proof of (iv) is same as in the proof of Proposition 3.3 in \cite{BSS12}.
Therefore, we prove only (iii).
Suppose that  $d(x_n,C)\rightarrow 0$ as $n\rightarrow \infty$.
Then for any $n,k\in\mathbb{N}$ and $c\in C$,
by the triangle inequality and the Fej\'{e}r monotonicity, we have
\[
d(x_{n+k},x_n)
\le d(x_{n+k},c)+d(x_n,c)
\le 2d(x_n,c),
\]
and the by taking infimum, we have
\begin{equation}\label{eqn:condition for Cauchy for Fejer monotone}
d(x_{n+k},x_n)
\le 2d(x_n,C),
\end{equation}
which implies that $\{x_n\}$ is a Cauchy sequence in $M$
and so $\{x_n\}$ converges to some $x\in M$.
Also, for any $n\in\mathbb{N}$, it holds that
\[
d(x,C)\le d(x,x_n)+d(x_n,C),
\]
which implies that $x\in C$.
The converse is obvious.
\end{proof}

Let $A$ and $B$ be closed convex subsets of a complete ${\rm CAT}(\kappa)$ space $(M,d)$.
The \textit{alternating projection method} produces a sequence $\{x_n\}$ by
\begin{equation}\label{eqn:alternating projection method}
x_{2n-1}=P_{A}(x_{2n-2}),\quad x_{2n}=P_{B}(x_{2n-1}),\qquad n \in \mathbb{N},
\end{equation}
where $x_0$ is a given starting point.

\begin{lemma}{\rm \cite{HFLL12}}\label{lmm:criterion of Delta-convergence of Feer monotone}
Let $(M,d)$ be a complete ${\rm CAT}(\kappa)$ space
and let $C\subset M$ be a non-empty set.
Suppose that the sequence $\{x_n\}\subset M$
is Fej\'er monotone with respect to $C$ and satisfies that $r(\{x_n\})<D_\kappa/2$.
Suppose also that any $\Delta$-cluster point $x$ of $\{x_n\}$ belongs to $C$.
Then $\{x_n\}$ $\Delta$-converges to a point in $C$.
\end{lemma}

The following lemma is a ${\rm CAT}(\kappa)$ space analogue of Lemma 3.4 in \cite{BSS12}.

\begin{lemma}\label{lem:Alternation Projections is Fejer monotone sequence wrt AinB}
Let $(M,d)$ be a complete ${\rm CAT}(\kappa)$ space with ${\rm diam}(M)<D_{\kappa}/2$.
Let $A$ and $B$ be non-empty convex closed subsets of $M$ with $A \cap B \neq\emptyset$.
Then the sequence $\{x_n\}$ given as in \eqref{eqn:alternating projection method}
constructed by the alternating projection method with a starting point $x_{0}$
is Fej\'{e}r monotone with respect to $A\cap B$.
\end{lemma}
\begin{proof}
Let $c \in A \cap B$. For fixed $n \in \mathbb{N}$,
without loss of generality we assume that $x_n \in A$. Note that $x_{n+1}=P_{B}(x_n)$.
If $x_{n+1}=c$, then the proof is clear, i.e., $d(x_{n+1},c)=0 \le d(x_n,c)$.
If $x_{n+1}\neq c$, by Proposition \ref{propo:inequality of projection},
we have
\begin{equation*}
c_{M}d(x_n, x_{n+1})^{2}+ d(x_{n+1},c)^{2} \le d(x_n, c)^{2},
\end{equation*}
which implies that $d(x_{n+1},c) \le d(x_n,c)$.
\end{proof}

Now, we recall the notion of asymptotically regularity for a sequence.
Let $(M,d)$ be a metric space.
A sequence $\{x_n\}$ in $M$ is said to be \textit{asymptotically regular}
if $\displaystyle \lim_{n\rightarrow 0} d(x_n,x_{n+1})=0$.
A rate of convergence of $\{d(x_n,x_{n+1})\}$ towards $0$ will be called a \textit{rate of asymptotic regularity}.

The next theorem gives us a rate of asymptotic regularity
of the sequence given as in \eqref{eqn:alternating projection method}
constructed by the alternating projection method in a ${\rm CAT}(\kappa)$ space.
We refer to Theorem 5.2 in \cite{Nicolae13} for a rate of asymptotic regularity of the alternating projections in a CAT$(0)$ space.
For the proof of Theorem \ref{thm:rate of asymptotic regularity of the alternating projection method},
we will follow and refine the proof of Theorem 5.2 in \cite{Nicolae13}.

\begin{theorem}\label{thm:rate of asymptotic regularity of the alternating projection method}
Let $(M,d)$ be a complete ${\rm CAT}(\kappa)$ space  with ${\rm diam}(M)<D_{\kappa}/2$.
Let $A$ and $B$ be non-empty convex closed subsets of $M$ with $A \cap B \neq\emptyset$.
Let $x_0$ be a starting point and $\{x_n\}$ be the sequence given as in \eqref{eqn:alternating projection method}.
Then for any $\epsilon >0$, there exists $N(\epsilon)\ge0$ such that for any $n \ge N(\epsilon)$,
it holds that
\[
d(x_n,x_{n+1}) \le \epsilon.
\]
More precisely, we can take $N(\epsilon)$ by
\[
N(\epsilon)=\left\{
              \begin{array}{ll}
                \left[\frac{D_{\kappa}^{2}}{4\epsilon c_{M}} \right], & \text{for}\,\,\epsilon < D_{\kappa} \\
                0, & \text{otherwise,}
              \end{array}
            \right.
\]
where $c_{M}$ is given as in \eqref{eqn:convexity for CAt(k) space} and $[a]$ is the largest integer less than or equal to $a$.
\end{theorem}
\begin{proof}
Let $c \in A \cap B$. Then by Lemma \ref{lem:Alternation Projections is Fejer monotone sequence wrt AinB}, we have
$d(x_{n+1},c) \le d(x_n,c)$ for all $n\in \mathbb{N}$. Since by assumption for $M$,
\[
d(x_n, x_{n+1}) \le d(x_n, c)+ d(c,x_{n+1})< D_{\kappa}.
\]
Hence the case of $\epsilon \ge D_{\kappa}$ is clear. Suppose that $\epsilon < D_{\kappa}$ and set
\begin{equation}\label{eqn:def of N}
N=N(\epsilon):=\left[\frac{D_{\kappa}^2}{4\epsilon c_{M}}\right].
\end{equation}
For fixed $n \in \mathbb{N}$, without loss of generality,
we assume that $x_n \in A$ and $x_{n+1}\notin A \cap B$.
Note that $x_{n+1}=P_{B}(x_n)$.
By Proposition \ref{propo:inequality of projection}, we have for $c \in A \cap  B$
\begin{equation}\label{eqn:distof xn,xn+1}
c_{M}d(x_n, x_{n+1})^{2} \le d(x_n, c)^{2}- d(x_{n+1},c)^{2}.
\end{equation}
If we assume that $d(x_n,x_{n+1})> \epsilon$ for all $n=1,\cdots,N$,
then by \eqref{eqn:distof xn,xn+1}, we have
\begin{equation*}
  c_{M}\sum_{i=1}^{N}d(x_i, x_{i+1})^2\le d(x_1,c)^2-d(x_{N+1},c)^2 < D_{\kappa}^2/4,
\end{equation*}
which implies that $c_{M} \epsilon N< D_{\kappa}^2/4$.
This contradicts to \eqref{eqn:def of N}.
Therefore, there exists $n \le N$ such that $d(x_n,x_{n+1}) \le \epsilon$.
But, the sequence $\{d(x_{n+1},x_n)\}$ is non-increasing. Indeed,
for fixed $n \in \mathbb{N}$, again without loss of generality,
we assume that $x_n \in A$. Then since $x_{n+2}=P_{A}(x_{n+1})$,
by Proposition \ref{pro:properties of metric projection}, we have
\begin{equation*}
  d(x_{n+1},x_{n+2})=d(x_{n+1},P_{A}(x_{n+1}))=d(x_{n+1},A) \le d(x_{n+1},x_n).
\end{equation*}
Therefore, the proof is completed.
\end{proof}

\begin{remark}\label{rem:to compare the result in [1]}
\upshape
A similar result to Theorem \ref{thm:rate of asymptotic regularity of the alternating projection method}
can be found in \cite{ARLAN15}.
However, in \cite{ARLAN15}, to obtain the similar result,
the firmly non-expensiveness of maps is assumed.
But for Theorem \ref{thm:rate of asymptotic regularity of the alternating projection method},
the non-expensiveness of metric projections is not necessary.
\end{remark}

\begin{lemma}\label{lem:maximum inequality}
Let $(M,d)$ be a complete ${\rm CAT}(\kappa)$ space with ${\rm diam}(M)<D_{\kappa}/2$.
Let $A$ and $B$ be non-empty convex closed subsets of $M$ with $A \cap B \neq\emptyset$.
Let $x_0$ be a starting point and $\{x_n\}$ be the sequence given as in \eqref{eqn:alternating projection method}
constructed by the alternating projection method.
Then for any $n \in \mathbb{N}$, it holds that
\begin{equation}\label{eqn:maximum inequality}
\max \left\{d(x_n,A)^{2},d(x_n,B)^{2} \right\} \le \frac{1}{c_{M}}\left( d(x_n, A \cap B)^{2}-d(x_{n+1}, A \cap B)^{2}\right),
\end{equation}
where $c_{M}$ is given as in \eqref{eqn:convexity for CAt(k) space}.
\end{lemma}
\begin{proof}
The proof is a modification of the first part of the proof of Theorem 4.1 in \cite{BSS12}.
For fixed $n \in \mathbb{N}$, without loss of generality,
we assume that $x_n \in A \notin A \cap B$
and $x_{n+1}=P_{B}(x_n)\notin A \cap B$.
Then by using Proposition \ref{propo:inequality of projection},
for any $z\in B$, we have
\begin{align*}
d(x_n,z)^{2}
&\ge d(z,P_{B}(x_n))^{2}+c_{M}d(x_n,P_{B}(x_n))^{2}\\
&=d(z,x_{n+1})^{2}+c_{M}d(x_n,x_{n+1})^{2},
\end{align*}
from which, by taking $z=P_{A \cap B}(x_n)\in A \cap B\subset B$,
we have
\[
d(x_n, P_{A \cap B}(x_n))^2\ge d(P_{A \cap B}(x_n),x_{n+1})^2+c_{M}d(x_n, x_{n+1})^2.
\]
Therefore, by Proposition \ref{pro:properties of metric projection},
it holds that
\[
d(x_n, A \cap B)^2\ge d(A \cap B,x_{n+1})^2+c_{M} d(x_n,B)^2.
\]
Similarly, in case of $x_n \in B$,
\[
d(x_n, A \cap B)^2\ge d(A \cap B,x_{n+1})^2+c_{M} d(x_n,A)^2.
\]
Hence, the proof is completed.
\end{proof}

Now, we recall the notion of regularity of sets in metric spaces (see \cite{BSS12}).
Let $(M,d)$ be a metric space and $A,B$ be subsets of $M$.
Then we say that
\begin{itemize}
  \item [(i)]  $A$ and $B$ are \textit{boundedly regular}
if for any bounded subset $S \subseteq M$ and any $\epsilon >0$,
there exists $\delta >0$ such that
for any $x \in S$ and $\max\{d(x,A),d(x,B)\} \le \delta$,
\[
d(x, A \cap B) \le \epsilon;
\]
  \item [(ii)] $A$ and $B$ are \textit{boundedly linearly regular} if
for any bounded subset $S \subseteq M$, there exists $k >0$ such that for $x \in S$,
\[
d(x, A \cap B) \le k \max\{d(x,A),d(x,B)\};
\]
  \item [(iii)] $A$ and $B$ are \textit{linearly regular}
  if there exists $k>0$ such that for any $x\in M$,
\[
d(x, A \cap B) \le k \max\{d(x,A),d(x,B)\}.
\]
\end{itemize}

If the metric space $(M,d)$ is bounded,
then boundedly regular and boundedly linearly regular are said to be regular
and linearly regular, respectively.
Since a ${\rm CAT}(\kappa)$ space with ${\rm diam}(M)<D_{\kappa}/2$ is bounded,
the notions of boundedly linearly regular and linearly regular are same.

The next theorem is the main result in this section.
\begin{theorem}\label{thm:convergence of Al.pro}
Let $(M,d)$ be a complete ${\rm CAT}(\kappa)$ space with ${\rm diam}(M)<D_{\kappa}/2$.
Let $A$ and $B$ be non-empty convex closed subsets of $M$ with $A \cap B \neq\emptyset$.
Let $x_0$ be a starting point and $\{x_n\}$ be the sequence given as in \eqref{eqn:alternating projection method}
constructed by the alternating projection method.
Then the following properties hold:
\begin{itemize}
  \item [\rm{(i)}] $\{x_n\}$ $\Delta$-converges to a point $x \in A \cap B$.
  \item [\rm{(ii)}] If $A$ and $B$ are boundedly regular, then $\{x_n\}$ converges to a point $x \in A \cap B$.
  \item [\rm{(iii)}] If $A$ and $B$ are boundedly linearly regular, then $\{x_n\}$ converges linearly to a point $x \in A \cap B$.
\end{itemize}
\end{theorem}
\begin{proof}
By Lemma \ref{lem:Alternation Projections is Fejer monotone sequence wrt AinB},
the sequence $\{x_n\}$ is Fej\'{e}r monotone with respect to $C:=A\cap B$.
Therefore, by (ii) in Proposition \ref{propo:properties of Fejer monotone sequence},
the sequence $\{d(x_n,C)\}$ is bounded and decreasing sequence in $\mathbb{R}$,
and so $\{d(x_n,C)\}$ converges to some point in $\mathbb{R}$.
Therefore, by \eqref{eqn:maximum inequality} in Lemma \ref{lem:maximum inequality},
we prove that
\begin{equation}\label{eqn:convergence of maixmaum of AB}
\max \left\{d(x_n,A),d(x_n,B) \right\} \rightarrow 0
\end{equation}
as $n \rightarrow \infty$.

(i) \enspace
Since $\{x_n\}$ is bounded with $r(\{x_n\})<D_{\kappa}/2$,
by (ii) in Proposition \ref{properties of bounded sequence},
$\{x_n\}$ has a $\Delta$-cluster point in $M$.
Let $x\in M$ be a $\Delta$-cluster point of $\{x_n\}$.
Then we take a subsequence $\{x_{n_k}\}$ of $\{x_n\}$ which $\Delta$-converges to $x$.
Then by Remark \ref{rem:weakly lsc at x for dC}
and \eqref{eqn:convergence of maixmaum of AB},
it holds that
\begin{align*}
d(x,A)=d(x,B)=0,
\end{align*}
which implies that $x\in A \cap B$.
Therefore, by Lemma \ref{lmm:criterion of Delta-convergence of Feer monotone},
we conclude that $\{x_n\}$ $\Delta$-converges to a point $x\in A \cap B$.

(ii) \enspace
Suppose that $A$ and $B$ are boundedly regular. Then since $\{x_n\}$ is a bounded sequence,
by \eqref{eqn:convergence of maixmaum of AB}, we see that
\[
d(x_n, A\cap B)\rightarrow 0
\]
as $n \rightarrow \infty$.
Therefore, by (iii) in Proposition \ref{propo:properties of Fejer monotone sequence},
$\{x_n\}$ converges to a point $x \in A\cap B$.

(iii) \enspace
Since $\{x_n\}$ is a bounded sequence,
and $A$ and $B$ are boundedly linearly regular,
there exists $k >0$ such that for all $n \in \mathbb{N}$,
\[
d(x_n, A \cap B) \le k \max\{d(x_n,A),d(x_n,B)\}.
\]
By \eqref{eqn:maximum inequality}, we have
\[
d(x_n, A \cap B)^2 \le \frac{k^2}{c_{M}} \left(d(x_n, A \cap B)^2-d(x_{n+1}, A \cap B)^2\right),
\]
which implies that
\[
d(x_{n+1}, A \cap B) \le \sqrt{1-\frac{c_{M}}{k^2}} d(x_n, A \cap B),
\]
where $c_{M}$ is given as in \eqref{eqn:convexity for CAt(k) space}.
Therefore, by (iv) in Proposition \ref{propo:properties of Fejer monotone sequence},
the proof of (iii) is completed.
\end{proof}

\begin{remark}
\upshape
As same as mentioned in Remark \ref{rem:to compare the result in [1]},
the result of Theorem 4.1 in \cite{ARLAN15}
similar to (i) in Theorem \ref{thm:convergence of Al.pro}
has been proved with the firmly non-expensiveness of maps.
\end{remark}

A metric space $(M,d)$ is said to be a \textit{boundedly compact}
if every bounded and closed subset of $M$ is compact.

\begin{corollary}\label{coro:strong convergence of Al.pro}
Let $(M,d)$ be a complete ${\rm CAT}(\kappa)$ space with ${\rm diam}(M)<D_{\kappa}/2$.
Let $A$ and $B$ be non-empty convex closed subsets of $M$ with $A \cap B \neq\emptyset$.
Let $x_0$ be a starting point and $\{x_n\}$ be the sequence given as in \eqref{eqn:alternating projection method}
constructed by the alternating projection method.
If $A$ or $B$ is boundedly compact,
then $\{x_n\}$ converges to a point $x \in A \cap B$.
\end{corollary}

\begin{proof}
The proof is similar to the proof of the second part of Theorem 4.1 in \cite{ARLAN15}.
By (i) in Theorem \ref{thm:convergence of Al.pro},
$\{x_n\}$ $\Delta$-converges to a point $x \in A \cap B$.
Since the sequence $\{x_n\}$ is Fej\'{e}r monotone with respect to $A\cap B$,
the sequence $\{d(x_n,x)\}$ is bounded and decreasing in $\mathbb{R}$,
and so $\{d(x_n,x)\}$ converges to a point in $\mathbb{R}$.
Without loss of generality, we assume that $A$ is boundedly compact.
Then $A$ is a compact subset in $M$.
Therefore, there exists a subsequence $\{x_{n_k}\}$ of $\{x_{2n-1} \}\subset A$
such that $\{x_{n_k}\}$ converges to a point $\widetilde{x} \in A$.
Thus, we have
\[
\lim_{k \rightarrow \infty}d(x_{n_k},\widetilde{x})
=0
\le\lim_{k\rightarrow \infty}d(x_{n_k},z)\,\, \text{for all}\,\,z \in M,
\]
which implies that $\widetilde{x} \in A(\{x_{n_k}\})$.
By the uniqueness of the asymptotic center, we have $x=\widetilde{x}$.
Since $\{d(x_n,x)\}$ converges, $\{x_n\}$ converges to $x \in A \cap B$.
\end{proof}

By applying the Hopf-Rinow Theorem (see \cite{BH99})
and simple modifications of the proof of Corollary \ref{coro:strong convergence of Al.pro},
we have the following corollary.

\begin{corollary}{\rm c.f. \cite{ARLAN15}}
Let $(M,d)$ be a complete ${\rm CAT}(\kappa)$ space with ${\rm diam}(M)<D_{\kappa}/2$.
Let $A$ and $B$ be non-empty convex closed subsets of $M$ with $A \cap B \neq\emptyset$.
Let $x_0$ be a starting point and $\{x_n\}$
be the sequence given as in \eqref{eqn:alternating projection method}
constructed by the alternating projection method.
If $A$ or $B$ is locally compact, then
$\{x_n\}$ converges to a point $x \in A \cap B$.
\end{corollary}
\begin{example}\rm
Let
\[
S=\left\{\left.\left[
                  \begin{array}{cc}
                    x & -y+iz \\
                    y+iz & x
                  \end{array}\right]\,\right|\, x,y,z \in \mathbb{R} \text{ with } x^2+y^2+z^2=1\right\}.
\]
Then we can identify $S \subseteq SU(2)$ with $\mathbb{S}^2$
by the map $\Phi:\mathbb{S}^2\to S$ defined as
\[
\Phi(x,y,z)= \left[
                  \begin{array}{cc}
                    x & -y+iz \\
                    y+iz & x
                  \end{array}\right].
\]
Consider the following six elements:
\begin{align*}
                 &M_{1}=\left[
                 \begin{array}{cc}
                    0 & i \\
                    i & 0
                  \end{array}\right],\qquad
                  M_{2}=\left[ \begin{array}{cc}
                    \frac{1}{2} & i\frac{\sqrt{3}}{2} \\
                    i\frac{\sqrt{3}}{2} & \frac{1}{2}
                  \end{array}
                \right],\qquad
               M_{3}=\left[
                  \begin{array}{cc}
                    0 & -\frac{1}{2}+i \frac{\sqrt{3}}{2} \\
                    \frac{1}{2}+i \frac{\sqrt{3}}{2} & 0 \\
                  \end{array}
                \right],\\
&M_{4}=\left[
                  \begin{array}{cc}
                    \frac{1}{3} & i\frac{2\sqrt{2}}{3} \\
                    i\frac{2\sqrt{2}}{3} & \frac{1}{3}
                  \end{array}\right],
M_{5}=                  \left[ \begin{array}{cc}
                    \frac{2}{3} & -\frac{1}{3} + i\frac{2}{3} \\
                    \frac{1}{3} + i\frac{2}{3} & \frac{2}{3}
                  \end{array}
                \right],
M_{6}=\left[
                  \begin{array}{cc}
                  \frac{2}{3} & -\frac{2}{3} + i\frac{1}{3}\\
                    \frac{2}{3} + i\frac{1}{3} & \frac{2}{3} \\
                  \end{array}
                \right]
\end{align*}
of $S$.
Then the sets
\begin{align*}
  A&=\overline{\rm{conv}}\left\{(0,0,1),(1/2,0,\sqrt{3}/2),(0,1/2,\sqrt{3}/2)\in \mathbb{S}^2 \right\},\\
  B&=\overline{\rm{conv}}\left\{(1/3,0,2\sqrt{2}/3),(2/3,1/3,2/3),(2/3,2/3,1/3)\in \mathbb{S}^2 \right\}
\end{align*}
can be considered as closed bounded convex sets generated by the subsets
\begin{align*}
  \widetilde{A}=\left\{M_{1},M_{2},M_{3}\right\}\quad \text{and}\quad
  \widetilde{B}=\left\{M_{4},M_{5},M_{6}\right\}
\end{align*}
of $S \subseteq SU(2)$, respectively,
and it is easy to see that the point $(1/3,0,2\sqrt{2}/3)$
is in the geodesic joining $(0,0,1)$ and $(1/2,0,\sqrt{3}/2)$
and so $A \cap B \neq\emptyset$.
Since the $2$-dimensional sphere $\mathbb{S}^2$ is a compact complete metric space,
$A$ and $B$ are compact subset in $\mathbb{S}^2$.
Thus, by Corollary \ref{coro:strong convergence of Al.pro},
the alternating sequence $\{x_n\}$ given as in \eqref{eqn:alternating projection method}
converges to a point $x \in A \cap B$.
Therefore, the sequence $\{\Phi(x_n)\}\subseteq SU(2)$ converges to a point $\Phi(x) \in \Phi(A) \cap \Phi(B) $
in the sense of a ${\rm CAT}(1)$ space.
\end{example}

\section*{Acknowledgments}
UCJ was supported by Basic Science Research Program
through the NRF funded by the MEST (NRF-2016R1D1A1B01008782).
YL was supported by the National
Research Foundation of Korea (NRF) grant funded by the Korea
government(MEST) (No.2015R1A3A 2031159).


\end{document}